\documentclass[12pt, reqno]{amsart}

\usepackage{amsthm,amssymb,amstext,amscd,amsfonts,amsbsy,amsrefs,amsxtra,latexsym,amsmath,xcolor,mathrsfs,fancybox,upgreek, soul,url}
\usepackage[top=0.94in, bottom=0.94in, left=1.25in, right=1.25in]{geometry}
\usepackage[english]{babel}
\usepackage[all,cmtip]{xy}
\usepackage[latin1]{inputenc}
\usepackage{cancel}
\usepackage{comment}
\usepackage{mdframed}
\usepackage{mathtools}
\usepackage{enumerate}
\usepackage{thmtools}
\usepackage{thm-restate}
\usepackage{enumitem}
\usepackage{etoolbox}
\usepackage{todonotes}
\usepackage{footmisc}
\usepackage{bigints}

\usepackage[colorlinks=true,linkcolor=black,citecolor=black,filecolor=black]{hyperref}

\newcommand{\leg}[2]{\genfrac{(}{)}{}{}{#1}{#2}}

\newtheorem{theorem}{Theorem}

\newtheorem{lemma}[theorem]{Lemma}

\newtheorem{proposition}[theorem]{Proposition}

\theoremstyle{definition}

\theoremstyle{remark}

\newtheorem*{remark}{Remark}
\newtheorem*{remarks}{Remarks}

\numberwithin{theorem}{section}
\numberwithin{proposition}{section}
\numberwithin{lemma}{section}
\numberwithin{corollary}{section}
\numberwithin{equation}{section}
\numberwithin{conjecture}{section}
\numberwithin{example}{section}

\setlist[enumerate,1]{before=}
\AfterEndEnvironment{enumerate}{}

\newcommand{\N}{\mathbb{N}}
\newcommand{\Z}{\mathbb{Z}}
\newcommand{\R}{\mathbb{R}}
\newcommand{\C}{\mathbb{C}}
\newcommand{\Q}{\mathbb{Q}}
\newcommand{\SL}{{\text {\rm SL}}}

\newcommand{\sgn}{\operatorname{sgn}}
\newcommand{\re}{\textnormal{Re}}

\def\H{\mathbb{H}}
\renewcommand{\pmod}[1]{\  \,  \left( \mathrm{mod} \,  #1 \right)}

\DeclareMathOperator{\tr}{tr}

\newcommand{\e}{\mathfrak{e}}
\renewcommand{\pmod}[1]{\  \,  \left( \mathrm{mod} \,  #1 \right)}

\DeclareMathOperator{\Gr}{Gr}
\newcommand{\CT}{\mathrm{CT}}
\DeclareMathOperator{\reg}{reg}
\DeclareMathOperator{\spt}{spt}

\begin{document}

\title[Higher Siegel theta lifts on Lorentzian lattices and Eichler--Selberg type relations]{Higher Siegel theta lifts on Lorentzian lattices, harmonic Maass forms, and Eichler--Selberg type relations}

\author{Joshua Males}
\address{Department of Mathematics, Machray Hall, University of Manitoba, Winnipeg,
	Canada}
\email{joshua.males@umanitoba.ca}

\begin{abstract}
We investigate so-called ``higher" Siegel theta lifts on Lorentzian lattices in the spirit of Bruinier--Ehlen--Yang and Bruinier--Schwagenscheidt. We give a series representation of the lift in terms of Gauss hypergeometric functions, and evaluate the lift as the constant term of a Fourier series involving the Rankin--Cohen bracket of harmonic Maass forms and theta functions. Using the higher Siegel lifts, we obtain a vector-valued analogue of Mertens' result stating that the Rankin--Cohen bracket of the holomorphic part of a harmonic Maass form of weight $\frac{3}{2}$ and a unary theta function, plus a certain form, is a holomorphic modular form. As an application of these results, we offer a novel proof of a conjecture of Cohen which was originally proved by Mertens, as well as a novel proof of a theorem of Ahlgren and Kim, each in the scalar-valued case.
\end{abstract}

\maketitle

\section{Introduction}
In recent years, there have been many investigations into certain theta lifts and their relationship to modular objects. Perhaps one of the most striking applications is in realising rationality and algebraicity results. In a recent breakthrough paper, Bruinier, Ehlen, and Yang \cite{bruinier2020greens} made a major advance towards the Gross--Zagier conjecture by proving that a certain two-variable Green function evaluated at CM points in one variable and an average over CM points in the other variable takes algebraic values. A pivotal result that the authors there used was the connection between the Green function and a ``higher" Siegel theta lift\footnote{Note that the authors there use a different signature convention than the current paper.}. A further example lies in \cite{ANBMS}, where Alfes-Neumann, Bringmann, Schwagenscheidt, and the author used the higher Siegel theta lift\footnote{The paper \cite{ANBMS} also used the higher Millson theta lift.} on lattices of signature $(1,2)$ to investigate traces of cycle integrals of a certain cusp form. By relating the lift to the Fourier coefficients of certain modular objects called harmonic Maass forms, the results also gave an alternative proof of the rationality of such traces of cycle integrals, which had previously been shown in \cite{alfes2018rationality}.

More recently, Bruinier and Schwagenscheidt \cite{BS} investigated the Siegel theta lift on a Lorentzian lattice $L$ (that is, of signature $(1,n)$). For $\tau = u+iv \in \H$  and $z \in \Gr(L)$, the Grassmanian of $L$ i.e.\@ positive lines in $L \otimes \R$, this is given by
\begin{align*}
\int_{\mathcal{F}}^{\reg} \left\langle f(\tau) , \overline{\Theta_L(\tau,z)} \right\rangle v^{k} d\mu(\tau),
\end{align*}
where the regularised integral is defined by $\int_{\mathcal{F} }^{\reg} \coloneqq \lim_{T \rightarrow \infty} \int_{\mathcal{F}_T}$, where $\mathcal{F}_T$ denotes the standard fundamental domain for $\Gamma$ truncated at height $T$. Here, $\mathcal{F}$ is the standard fundamental domain for $\SL_2(\Z)$, $f$ is a vector-valued harmonic Maass form of appropriate weight, $\langle \cdot , \cdot \rangle$ is the natural bilinear pairing on the group ring $\C[L'/L]$ with $L'$ the dual lattice, and $d\mu= \frac{dudv}{v^2}$ is the usual invariant measure on $\H$. Moreover, $\Theta_L(\tau,z)$ is a vector-valued Siegel theta function defined on $\H\times \Gr(L)$. As a function of $\tau \in \H$, it transforms as a vector-valued modular form of weight $(1-n)/2$ with respect to the Weil representation $\rho_L$ associated to $L$, and is invariant in $z$ under the subgroup of $O(L)$ that fixes $L'/L$ pointwise. Explicitly, this means that for $(g,\phi)  \in \text{Mp}_2(\Z)$, the metaplectic extension of $\SL_2(\Z)$ (see Section \ref{Section: Weil representation}) and $h$ in the subgroup of of $O(L)$ that fixes $L'/L$ pointwise we have
\begin{align*}
	\Theta_L(g\tau,z) = \phi(\tau) \overline{\phi(\tau)}^{n} \rho_L(g,\phi) \Theta_L(\tau,z), \qquad \Theta_L(\tau, hz) = \Theta_L(\tau,z).
\end{align*}
 
In the present paper we consider an extension of this lift, first considered by Bruinier, Ehlen, and Yang in \cite{bruinier2020greens} for the $n=2$ case. Borrowing their terminology, we call this a higher Siegel theta lift. Let $k \coloneqq \frac{1-n}{2}$. The case of $j=0$ was studied centrally in the influential work of Bruinier, Imamo\={g}lu, and Funke \cite{BIF}. Since the case $j=0$ is well-understood, let $j \in \N$ throughout unless otherwise stated. For a harmonic Maass form $f \in H_{k-2j,L}$ (see Section \ref{Sec: harmonic M forms} for a precise definition) we consider
\begin{align*}
\Lambda^{\reg}_j \left(f, z\right)
\coloneqq
\int_{\mathcal{F}}^{\reg} \left\langle R_{k-2j}^{j}(f)(\tau) , \overline{\Theta_L(\tau,z)} \right\rangle v^{k} d\mu(\tau),
\end{align*}
where  $R_{\kappa}^{n} \coloneqq R_{\kappa+2n-2}\circ \dots \circ R_{\kappa}$ with $R_{\kappa}^0 \coloneqq \mathrm{id}$ is an iterated version of the Maass raising operator $R_{\kappa} \coloneqq 2i\frac{\partial}{\partial \tau}+ \frac{\kappa}{v}$. By the results of \cite[Section~2.3]{bruinier2004borcherds} for the Siegel theta function the integral converges for every $z \in \H$. 

This situation lies at the interface of \cite{bruinier2020greens,ANBMS} which included iterated raising operators but restricted to the case of $n=2$, and \cite{BS} where no iterated raising operators were included, but $n$ was unrestricted. We borrow techniques from each of these papers, and in Section \ref{Section: the lift} obtain a description of the lift $\Lambda^{\reg}$ as a series involving Gauss hypergeometric functions ${}_2F_1$, as well as a constant term of a Fourier expansion involving coefficients of $f$, a theta function, and a harmonic Maass form. We also note how one could obtain the Fourier expansion of the lift, but omit the calculation as it is not required here. Note that for $j=0$ we recover the lift studied by Bruinier \cite{bruinier2004borcherds} and Bruinier--Schwagenscheidt \cite{BS}, and for $n=2$ the lift used centrally in \cite{ANBMS}.

Moreover, in \cite[Theorem 1.2]{Mer1} Mertens used holomorphic projection to show that the Fourier coefficients of all scalar-valued mock modular forms of weight $\frac{3}{2}$ satisfy Eichler-Selberg type relations. We provide a vector-valued analogue of these relations, relying on the framework of the higher Siegel theta lifts for $j >0$. To state the result, we need some notation. When evaluated at a special point $w$ (see Section \ref{Sec: theta functions}) the Siegel theta function on $L$ splits as the tensor product $\Theta_{P} \otimes \Theta_N$ where $P$ is one-dimensional positive-definite, and $N$ is $n$-dimensional negative-definite. For $X \in L$ by $X_z$ we mean the projection of $X$ onto $z$, and furthermore we denote vector-valued Rankin--Cohen brackets by $[\cdot,\cdot]_\ell$. Finally, let $\mathcal{G}_P^+$ be the holomorphic part of a preimage of $\Theta_P$ under $\xi_{\kappa} \coloneqq 2iv^\kappa \overline{\frac{\partial}{\partial \overline{\tau}}}$. By computing the lift in two different ways, taking the difference and invoking Serre duality, we prove the following theorem (for undefined notation see Section \ref{Section: prelims}).
\begin{theorem}\label{Thm: MAIN}
	Let $L$ be an even lattice of signature $(1,1)$ and $w$ be a special point. Then for every $j > 0$ there is an explicit constant $C_j \in \R$ such that	
	\begin{align*}
	\left[\mathcal{G}_P^+(\tau), \Theta_{N^-}(\tau) \right]_j + C_j  \sum_{m =0}^{\infty} \sum_{\substack{X \in L'\\ q(X)=m}} \left(\sqrt{\lvert q(X_{w^\perp}) \rvert} - \sqrt{\lvert q(X_{w}) \rvert}\right)^{1+2j} q^m
	\end{align*}
	is a holomorphic vector-valued modular form of weight $2j+2$ for $\rho_L$.
\end{theorem}

\begin{remarks}
	\begin{enumerate}
		\item One could obtain a similar vector-valued analogue of \cite[Theorem 1.3]{Mer1}, i.e.\@ for $\mathcal{G}_P^+$ replaced by the holomorphic part of harmonic Maass forms of weight $\frac{1}{2}$ (so including the classical mock theta functions of Ramanujan) by similar methods, replacing $\Theta_L$ by a modified theta function $\Theta_L^*$ as in \cite[Section 4.3]{BS}. In doing so, one could also obtain new recurrence relations for the classical mock theta functions. We leave the details for the interested reader.
		
		\item For $n \geq 2$ it is possible to obtain similar theorems. However, the form on the right-hand side becomes much more complicated. Using the contiguous Gauss hypergeometric functions to rewrite the output of Theorem \ref{Theorem: lift as series}, the forms should be linear combinations of those stated in Theorem \ref{Thm: MAIN} and similarly-shaped sums with factors of $\arcsin\left( \sqrt{\frac{q(X)}{q(X_{z^\perp})}} \right)$ (compare \cite[equation (3.3)]{BS}).
	\end{enumerate}
	
\end{remarks}

As an application, we obtain novel proofs of a conjecture of Cohen \cite{Cohen} which was originally proved by Mertens \cite[Theorem 1]{Mer}, as well as a theorem of Ahlgren and Kim \cite[Proposition 4]{AhlKim}, each of which is detailed below. For $\ell \geq 1$ and $n \geq 1$ we let
\begin{align*}
\lambda_\ell(n) = \frac{1}{2}\sum_{d \mid n}\min(d,n/d)^\ell,
\end{align*}
and
\begin{align*}
\mathcal{H}(\tau) = \sum_{n \geq 0}H(n)q^n,
\end{align*}
where for $n \in \N$, we let $H(n)$ be the Hurwitz class numbers, i.e.\@ the class number of binary quadratic forms of discriminant of discriminant $-n$, where the class containing $a(x^2+y^2)$ (resp.\@ $a(x^2+xy+y^2)$) is counted with multiplicity $\frac{1}{2}$ (resp.\@ $\frac{1}{3}$), and we set $H(0) \coloneqq -\frac{1}{12}$. Then $\mathcal{H}$ is a mock modular form of weight $\frac{3}{2}$, i.e. the holomorphic part of a scalar-valued harmonic Maass form of weight $\frac{3}{2}$. Let $\theta(\tau) \coloneqq \sum_{n \in \Z}q^{n^2}$ be the classical Jacobi theta function. For $f = \sum_n c_f(n)q^n$ we let $f^{\text{odd}} = \sum_{n \text{ odd}}c_f(n)q^n$.
Then we have the following theorem of Mertens \cite[Theorem 1]{Mer}. We offer a novel proof in Section \ref{Sec: mertens}.
\begin{theorem}\label{Thm: mertens}
	For $j \geq 0$ the function
	\[
	c_j[\mathcal{H},\theta]_j^\text{odd} + \sum_{n \geq 0}\lambda_{2j+1}(2n+1) q^{2n+1}
	\]
	with $c_j \coloneqq \frac{j! \sqrt{\pi}}{\Gamma\left(j+\frac{1}{2}\right)}$ is a holomorphic modular form of weight $2j + 2$ for $\Gamma_0(4)$. For $j > 0$ it is a cusp form.
\end{theorem}

Furthermore, let $\spt$ be the classical smallest parts partition function of Andrews \cite{And} and $p(n)$ the integer partition function. Define 
\begin{align*}
g(\tau) = q^{-\frac{1}{24}} \sum_{n \geq 0} \left( \spt\left(n\right) +\frac{1}{12}(24n-1)p(n) \right) q^{n}.
\end{align*}
Let $\eta$ be the usual Dedekind $\eta$-function and define
\begin{align*}
G_k (\tau) \coloneqq -\sum_{r>s>0} \leg{12}{r^2-s^2} s^{k-1} q^{\frac{rs}{6}}.
\end{align*}
Then Zagier stated \cite[Section 6]{Zag} and Ahlgren and Kim proved \cite[Proposition 1.4]{AhlKim} the following theorem\footnote{For the $j=0$ case Ahlgren and Kim explicitly showed equality with the quasi-modular Eisenstein series $\frac{1}{12}E_2$.}. Again, we offer a novel proof in Section \ref{Sec: mertens}.
\begin{theorem}\label{Thm: AK}
	For every $j > 0$ we have that
	\begin{align*}
	G_{2j+2}(\tau) + 24^j \binom{2j}{j}^{-1} \left[g\left(\frac{\tau}{24}\right),\eta(\tau) \right]_j
	\end{align*}
	is a modular form of weight $2j+2$ on $\SL_2(\Z)$.
\end{theorem}

To end the introduction, we remark as a cute application that by choosing $n=j=1$ in Theorem \ref{Thm: MAIN} one obtains a re-proof in terms of higher Siegel lifts of the classical relation
\begin{align*}
\sum_{\substack{n \in \Z }} \left(t-n^2\right) H\left(4t-n^{2}\right) =  \sum_{\substack{a,b \in \N \\ab =t}} \min\left(a,b \right)^{3}
\end{align*}
that may also be concluded by Theorem \ref{Thm: mertens}. One may obtain proofs of other similar relations of Hurwitz class numbers.

\subsection*{Outline}
We begin in Section \ref{Section: prelims} by recalling preliminary results needed for the rest of the paper. In Section \ref{Section: the lift} we prove the main results of the paper regarding the higher Siegel lift.  Finally, Section \ref{Sec: mertens} is devoted to proving Theorem \ref{Thm: MAIN} and giving novel proofs of Theorems \ref{Thm: mertens} and \ref{Thm: AK}.

\subsection*{Acknowledgments}
The author thanks Markus Schwagenscheidt for many insightful conversations on the contents of the paper, in particular suggesting the connection to Theorems \ref{Thm: mertens} and \ref{Thm: AK}, as well as useful comments on previous versions of the paper. The author would also like to thank Jan Bruinier and Andreas Mono for helpful comments on an earlier draft of the paper. The research conducted for this paper is supported by the Pacific Institute for the Mathematical Sciences (PIMS). The research and findings may not reflect those of the Institute.

\section{Preliminaries}\label{Section: prelims}
We collect some preliminary results needed for the sequel.
\subsection{The Weil representation}\label{Section: Weil representation}
The metaplectic extension of $\SL_2(\Z)$ is defined as
\begin{equation*}
\widetilde{\Gamma} \coloneqq \text{Mp}_2(\Z) \coloneqq \left\{ (\gamma, \phi) \colon \gamma = \left(\begin{matrix}
a & b \\ c & d
\end{matrix}\right)\in \SL_2(\Z), \phi\colon \H \rightarrow \C \text{ holomorphic}, \phi^2(\tau) = c\tau+d  \right\},
\end{equation*}
with generators $\widetilde{T} := \left(\left( \begin{smallmatrix} 1 & 1 \\ 0 & 1 \end{smallmatrix} \right),1\right)$ and $\widetilde{S} := \left(\left( \begin{smallmatrix} 0 & -1 \\ 1 & 0 \end{smallmatrix}\right) ,\sqrt{\tau}\right)$. Moreover, we let $\widetilde{\Gamma}_\infty$ be the subgroup generated by
$\widetilde{T}$.

Let $L$ be an even lattice of signature $(r,s)$ with quadratic form $q$ and associated bilinear form $(\cdot,\cdot)$. We let $L'$ denote its dual lattice, which has a group ring of $L'/L$ denoted by $\C[L'/L]$ with standard basis elements $\e_{\mu}$ for $\mu \in L'/L$. Furthermore, there is a natural bilinear form  $\langle\cdot,\cdot\rangle$ on the group ring given by $\langle\e_{\mu},\e_{\nu}\rangle = \delta_{\mu,\nu}$. Then the Weil representation $\rho_L$ associated with $L$ is the representation of $\widetilde{\Gamma}$ on $\C[L'/L]$ defined by
\begin{align*}
\rho_L\left(\widetilde{T}\right)(\e_\mu) \coloneqq e(q(\mu)) \e_\mu, \qquad
\rho_L\left(\widetilde{S}\right)(\e_\mu) \coloneqq \frac{e\left(\frac18(s-r)\right)}{\sqrt{|L'/L|}} \sum_{\nu \in L'/L} e(-(\nu,\mu)) \e_{\nu}.
\end{align*}
Here and throughout $e(x) \coloneqq e^{2\pi i x}$. The Weil representation $\rho_{L^-}$ associated to the lattice $L^- = (L,-q)$ is called the dual Weil representation of $L$. 

%We also make use of the sublattice
%\begin{align*}
%L_0'\coloneqq \left\{ \lambda \in L' \colon (\lambda,\ell) \equiv 0 \pmod{N}  \right\},
%\end{align*}
%where $N$ is the unique integer such that $(\ell,L)= N \Z$.

\subsection{Harmonic Maass forms}\label{Sec: harmonic M forms}
Let $\kappa \in \frac{1}{2} \Z$ and define the usual slash-operator by
\[
f\mid_{\kappa,\rho_{L}}(\gamma,\phi) (\tau) \coloneqq \phi(\tau)^{-2\kappa}\rho_{L}^{-1}(\gamma,\phi)f(\gamma\tau),
\]
for a function $f\colon \H \rightarrow \C[L'/L]$ and $(\gamma,\phi) \in \widetilde{\Gamma}$. A harmonic Maass form of weight $\kappa$ with respect to $\rho_L$ is any smooth function $f \colon \H \rightarrow \C[L'/L]$ that satisfies the following three conditions \cite{bruinierfunke2004}.
\begin{enumerate}[wide, labelwidth=!, labelindent=0pt]
	\item  It vanishes under the action of the weight $\kappa$ Laplace operator
	\[
	\Delta_{\kappa} \coloneqq -v^2\left(\frac{\partial^2}{\partial u^2} + \frac{\partial^2}{\partial v^2} \right) + i\kappa v\left(\frac{\partial}{\partial u} + i\frac{\partial}{\partial v} \right),
	\]
	\item It is invariant under the slash-operator $\mid_{\kappa,\rho_L}$,
	\item There exists a $\C[L'/L]$-valued Fourier polynomial (the principal part of $f$)
	\begin{equation*}
	P_f(\tau) \coloneqq \sum_{\mu \in L'/L} \sum_{n \leq 0} c_f^+(\mu,n) e(n\tau) \e_\mu 
	\end{equation*}
	such that $f(\tau) - P_f(\tau) = O(e^{-\varepsilon v})$ as $v \rightarrow \infty$ for some $\varepsilon > 0$.
\end{enumerate}
 The vector space of harmonic Maass forms of weight $\kappa$ with respect to $\rho_L$ is denoted by $H_{\kappa,L}$, and $M_{\kappa,L}^!$ is the subspace of weakly holomorphic modular forms. It is a classical fact that any $f \in H_{\kappa,L}$ can be decomposed as $f = f^+ + f^-$ where $f^+$ is holomorphic and $f^-$ is non-holomorphic. These have Fourier expansions of the form
\begin{align*}
f^+(\tau) &= \sum_{\mu \in L'/L} \sum_{n \gg -\infty} c_f^+(\mu,n) e(n\tau) \e_\mu, \quad
f^-(\tau) = \sum_{\mu \in L'/L} \sum_{n < 0} c_f^-(\mu,n)\Gamma(1-\kappa, 4\pi |n| v) e(n\tau) \e_\mu,
\end{align*}
where $\Gamma(s,x): = \int_x^\infty t^{s-1}e^{-t} dt$ denotes the incomplete Gamma function.

The antilinear differential operator $\xi_\kappa = 2iv^\kappa \overline{\frac{\partial}{\partial {\overline{\tau}}}}$ from the introduction maps a harmonic Maass form $f \in H_{\kappa,L}$ to a cusp form of weight  $2-\kappa$ for $\rho_{L^-}$. In addition to the Maass raising operator in the introduction, we also require the Maass lowering operator $L_\kappa \coloneqq -2iv^2\frac{\partial}{\partial \overline{\tau}}$.

\subsection{Maass--Poincar\'{e} series}\label{Section: Maass--Poincare series}
Let $\kappa \in \frac{1}{2} \Z$ with $\kappa < 0$. We let $M_{\mu, \nu}$ be the usual $M$-Whittaker function - see \cite[equation~13.1.32]{abramowitz1988handbook}, for example - and for $s \in \C$ and $y \in \R \backslash \{0\}$, define
\begin{align}\label{curlyM}
\mathcal{M}_{\kappa,s} (y) \coloneqq |y|^{-\frac{\kappa}{2}} M_{\sgn(y) \frac{\kappa}{2}, s-\frac{1}{2}} (|y|).
\end{align}
For $\mu \in L'/L$ and $m \in \Z - q(\mu)$ with $m > 0$ we define the vector-valued Maass--Poincar\'{e} series \cite{bruinier2004borcherds}
\begin{align*}
F_{\mu, -m, \kappa,s}(\tau) \coloneqq \frac{1}{2 \Gamma(2s)} \sum_{(\gamma,\phi) \in \widetilde{\Gamma}_\infty \backslash \widetilde{\Gamma}} \left(\mathcal{M}_{\kappa,s} (-4 \pi m v) e(-mu) \e_\mu\right) \mid_{\kappa, \rho_L} (\gamma,\phi)(\tau).
\end{align*}
The series converges absolutely for $\re(s) > 1$, and at the point $s = 1-\frac{\kappa}{2}$, the function
\[
F_{\mu, -m, \kappa}(\tau) \coloneqq F_{\mu,-m,\kappa,1-\frac{\kappa}{2}}(\tau)
\] 
defines a harmonic Maass form in $H_{\kappa,L}$ with principal part $e(-m\tau)(\e_{\mu}+\e_{-\mu})+\mathfrak{c}$ for some constant $\mathfrak{c} \in \C[L'/L]$. Let $H_{\kappa,L}^!$ be the space of harmonic Maass forms with at most linear exponential growth at $\infty$, and $H^{\text{cusp}}_{\kappa,L}$ be the subspace of $H_{\kappa,L}^!$ that maps to a cusp form under $\xi_\kappa$. Then it is a classical fact that for $\kappa <0$ any $f \in H_{\kappa,L}^{\text{cusp}}$ may be written as
\begin{align*}
f(\tau)= \frac{1}{2} \sum_{h \in L'/L} \sum_{m \geq 0} c_f^+(h,-m) F_{h,m,\kappa}(\tau) + \mathfrak{c},
\end{align*}
where $\mathfrak{c}$ is a constant in $\C[L'/L]$ which may be non-zero only if $\kappa=0$.

The following lemma is \cite[Lemma 2.1]{ANBMS}, and follows inductively from \cite[Proposition~3.4]{bruinier2020greens}.
\begin{lemma}\label{Lemma: R_k^n acting on F_m,h}
	For $n \in \N_0$ we have that
	\begin{equation*}
	R_\kappa^n \left(F_{\mu, -m, \kappa,s}\right)(\tau) = (4\pi m)^n \frac{\Gamma\left(s+n+\frac{\kappa}{2}\right)}{\Gamma\left(s+\frac{\kappa}{2}\right)} F_{\mu, -m,\kappa+2n,s}(\tau).
	\end{equation*}
\end{lemma}

\subsection{Operators on vector-valued modular forms}\label{Section: representations and lattices} Let $L$ be an even lattice. Then the space of $\C[L'/L]$-valued smooth modular forms of weight $\kappa$ with respect to the representation $\rho_L$ is denoted by $A_{\kappa,L}$. 

Let $K \subset L$ be a sublattice of finite index. Since $K \subset L \subset L' \subset K'$ we have $L/K \subset L'/K \subset K'/K$, and therefore obtain a map $L'/K \rightarrow L'/L, \mu \mapsto \bar{\mu}$. For $\mu \in K'/K$ and $f \in A_{\kappa, L}$,  and $g \in A_{\kappa,K}$, we define
\begin{equation*}
(f_K)_\mu := \begin{dcases}
f_{\bar{\mu}}& \text{ if } \mu \in L'/K, \\
0& \text{ if } \mu \not\in L'/K,
\end{dcases}\qquad 
\left(g^L\right)_{\bar{\mu}} = \sum_{\alpha \in L/K} g_{\alpha + \mu},
\end{equation*}
where $\mu$ is a fixed preimage of $\bar{\mu}$ in $L'/K$. The following lemma may be found in \cite[Section~3]{bruinierFaltings}.

\begin{lemma}\label{Lemma: restriction and trace maps}
	There are two natural maps
	\begin{equation*}
	\textup{res}_{L/K} \colon A_{\kappa, L} \rightarrow A_{\kappa, K}, \hspace{10pt} f \mapsto f_K,
	\qquad\qquad
	\tr_{L/K} \colon A_{\kappa, K} \rightarrow A_{\kappa, L}, \hspace{10pt} g \mapsto g^L,
	\end{equation*}
	such that for any $f \in A_{\kappa,L}$ and $g \in A_{\kappa,K}$, we have $
	\langle f, \overline{g}^L \rangle = \langle f_K, \overline{g} \rangle$.
\end{lemma}

\subsection{Rankin--Cohen brackets}
Let $K$ and $L$ be even lattices. For $n\in\N_0$ and functions $f \in A_{\kappa,K}$ and $g \in A_{\ell,L}$ with $\kappa,\ell \in \frac{1}{2} \Z$ the $n$-th Rankin--Cohen bracket is defined by
\begin{align*}\label{Definition: RC brackets}
[f,g]_n \coloneqq \frac{1}{(2\pi i)^n}\sum_{\substack{r,s \geq 0\\ r+s=n}} (-1)^r \frac{\Gamma(\kappa+n) \Gamma(\ell+n)}{\Gamma(s+1) \Gamma(\kappa+n-s) \Gamma(r+1) \Gamma(\ell+n-r)}  f^{(r)} \otimes g^{(s)},
\end{align*}
where $f^{(r)} = \frac{\partial^r}{\partial \tau^r} f$. For two vector-valued functions $f = \sum_{\mu}f_\mu \e_\mu \in A_{\kappa,K}$ and $g = \sum_{\nu}g_\nu \e_\nu \in A_{\ell,L}$ their tensor product is defined as 
\[
f \otimes g \coloneqq \sum_{\mu , \nu} f_\mu g_\nu \e_{\mu + \nu} \in A_{\kappa+\ell,K \oplus L}.
\]
The following proposition is the vector-valued version of \cite[Proposition~3.6]{bruinier2020greens}.

\begin{proposition}\label{Proposition: lowering of RC brackets}
	Let $f \in H_{\kappa,K}$ and $g \in H_{\ell,L}$ be harmonic Maass forms. For $n \in \N_0$ we have
	\begin{align*}
	(-4\pi)^n L_{\kappa+\ell+2n}\left([f,g]_n \right)= \frac{\Gamma(\kappa+n) }{\Gamma(n+1) \Gamma(\kappa)} L_\kappa (f) \otimes R_\ell^n (g) + (-1)^n \frac{ \Gamma(\ell+n)}{  \Gamma(n+1) \Gamma(\ell)} R_\kappa^n (f) \otimes L_\ell (g).
	\end{align*}
\end{proposition}

\subsection{Theta functions}\label{Sec: theta functions}
Let $(K,q)$ be a positive-definite lattice of rank $n$. Then one may define the vector-valued theta function
\begin{equation*}
\Theta_K(\tau) \coloneqq \sum_{\mu \in K'/K} \sum_{X \in K +\mu} e(q(X) \tau) \e_\mu,
\end{equation*}
which is a holomorphic modular form of weight $\frac{n}{2}$ for the Weil representation $\rho_K$.

For an even lattice $L$ of signature $(1,n)$ the Siegel theta function is defined by
\begin{align}\label{Definition: Siegel theta}
\Theta_L(\tau,z) \coloneqq v^{\frac{n}{2}} \sum_{\mu \in L'/L} \sum_{X \in L+\mu} e\left( q(X_z) + q(X_{z^\perp}) \right) \mathfrak{e}_\mu,
\end{align}
where $X_z$ denotes the orthogonal projection of $X$ onto $z$. The Siegel theta function transforms in $\tau$ like a modular form of weight $\frac{1-n}{2}$ for  $\rho_L$ and is invariant in $z$ under the subgroup $\Gamma_L$ of the orthogonal group $\mathrm{O}(L)$ which fixes the classes of $L'/L$. If $K \subset L$ is a sublattice of finite index, Lemma~\ref{Lemma: restriction and trace maps} implies that
\begin{align}\label{Equation: trace map for Theta_L}
\Theta_L = (\Theta_K)^L.
\end{align}

As in \cite{BS}, we call $w \in \Gr(L)$ a special point if it is defined over $\Q$ (i.e. $w \in L \otimes \Q$). Then the orthogonal complement of $w$ in the non-degenerate space $V = L \otimes \R$ is defined over $\Q$ as well, and we denote it by $w^\perp$. We have the splitting $L \otimes \Q = w \oplus w^\perp$. Let $P = L \cap w$ and $N = L \cap w^\perp$ be the corresponding positive definite one-dimensional and negative definite $n$-dimensional sublattices of $L$, and note that $P \oplus N$ has finite index in $L$. A direct computation shows that the evaluation of the Siegel theta function $\Theta_L$ at $w$ splits as
\begin{align}\label{Equation: splitting}
\Theta_L(\tau,w) = \Theta_{P \oplus N}(\tau,w) = \Theta_P(\tau) \otimes v^{\frac{n}{2}} \overline{\Theta_{N^-}(\tau)}.
\end{align}

\subsection{Unary theta functions }\label{Sec: unary theta, atkin-lehner}
Let $d \in \N$. Then the lattice $\Z(d) = (\Z,dx^2)$ is one-dimensional positive-definite, has level $4d$, and its discriminant group is isomorphic to $\Z/2d\Z$ with the quadratic form $x \mapsto x^2/4d$. We define the unary theta series by
\begin{align*}
\theta_{\frac{1}{2}, d} (\tau) \coloneqq \sum_{r \pmod{2d}} \sum_{\substack{n \in \Z \\ n \equiv r \pmod{2d}}} q^{\frac{n^2}{4d}} \mathfrak{e}_r.
\end{align*}
It is a holomorphic modular form of weight $\frac{1}{2}$ for the Weil representation of $\Z(d)$.

Given an exact divisor $c \mid\mid d$, we can associate an Atkin-Lehner involution $\sigma_c$ on $\Z/2d\Z$, defined by the equations $\sigma_c(x)\equiv -x \pmod{2c}$ and $\sigma_c(x) \equiv x \pmod{\frac{2d}{c}}$. These involutions act on vector-valued modular forms for the Weil representation associated to $\Z(d)$ via the action
\begin{align*}
\left(\sum_{r\pmod{2d}} f_r(\tau) \mathfrak{e}_r\right)^{\sigma_c} = \sum_{r \pmod{2d}} f_{\sigma_c(r)}(\tau) \mathfrak{e}_r.
\end{align*}

\subsection{Serre duality}
Let $L$ be an even lattice, and let $k \in \frac{1}{2}\Z$. Then we have the following proposition (compare e.g. \cite[Proposition 2.5]{LiSch}).

\begin{proposition}[Serre duality]\label{Prop: Serre duality} Let 
	\[
	g(\tau) = \sum_{h \in L'/L}\sum_{n \geq 0}c_g(h,n)q^n\e_h
	\]
	be a holomorphic $q$-series which is bounded at $i\infty$. Then $g(\tau)$ is a holomorphic modular form of weight $k$ for the Weil representation $\rho_L$ if and only if
	\[
	\{g,f \}: =\sum_{h \in L'/L}\sum_{n \geq 0}c_g(h,n)c_f(h,-n) = 0
	\]
	for every weakly holomorphic modular form $f$ of weight $2-k$ for $\overline{\rho}_L$.
	\end{proposition}

\section{The higher Siegel theta lift}\label{Section: the lift}
In this section we compute the lift $\Lambda_j^{\reg}(f,z)$ for certain input functions $f$ in two different ways (as well as offering a pathway as to how one could compute its Fourier expansion).

\subsection{A series representation}

We first obtain an expression for $\Lambda^{\reg}_j \left(f, z\right)$ as a series whenever $f \in H^{\text{cusp}}_{k-2j,L}$. Recall that $k=\frac{1-n}{2}$.
\begin{theorem}\label{Theorem: lift as series}
	Assume that $q(X_z) \neq 0$ for every $X \in L+\mu$ with $q(X) = -m$. 	Let $f \in H^{\text{cusp}}_{k-2j,L}$. For every $j \in \N$, we have
	\begin{multline*}
	\Lambda^{\reg}_j (f,z) = (4\pi)^{\frac{1}{2}+j} \frac{ \Gamma(1+j)\Gamma(\frac{n}{2}+j)}{\Gamma\left(\frac{n+3}{2}+2j\right)}  \sum_{\substack{ X \in L'\\ q(X)<0}} c_f^+(X,q(X))  q(X)^{\frac{n+1}{2}+2j}  \left(- q(X_{z^\perp})\right)^{-\frac{n}{2}-j} \notag \\  \times {}_2F_1\left(\frac{n}{2}+j, 1+j;\frac{n+3}{2}+2j;\frac{q(X)}{q(X_{z^\perp})}\right).
	\end{multline*}
	where ${}_2F_1$ denotes the classical Gauss hypergeometric function. The series on the right-hand side converges absolutely.
\end{theorem}
\begin{proof}
	We consider the regularized theta lift of the Maass--Poincar\'e series $F_{\mu,-m,k-2j,s}$. Applying Lemma~\ref{Lemma: R_k^n acting on F_m,h} we obtain
	\begin{equation*}
	\Lambda^{\text{reg}} \left(F_{\mu,-m,k-2j,s}, z\right) 
	= (4\pi m)^{j} \frac{\Gamma\left(s+\frac{k}{2}\right)}{\Gamma\left(s+\frac{k}{2}-j\right)} \int_{\mathcal{F}}^{\reg} \left\langle  F_{\mu,-m,k,s}\left(\tau\right), \overline{\Theta_L (\tau,z)} \right\rangle v^{k} d\mu(\tau).
	\end{equation*}
	The usual unfolding argument yields the above expression as
	\begin{align*}
	2(4\pi m)^{j} \frac{\Gamma\left(s+\frac{k}{2}\right)}{\Gamma(2s) \Gamma\left(s+\frac{k}{2}-j\right)} \int_{0}^{\infty} \int_{0}^{1} \mathcal{M}_{k,s}(-4 \pi m v)e(-mu)   \overline{\Theta_{L,\mu}( \tau, z)} v^{k-2} dudv,
	\end{align*}
	where $\Theta_{L,\mu}$ denotes the $\mu$-th component of $\Theta_L$. Inserting \eqref{Definition: Siegel theta} and \eqref{curlyM}, after evaluating the integral over $u$ we are left with
	\begin{align*}
	2(4\pi m)^{j-\frac{k}{2}} \frac{\Gamma\left(s+\frac{k}{2}\right)}{\Gamma(2s) \Gamma\left(s+\frac{k}{2}-j\right)} \sum_{\substack{X \in L+\mu \\ q(X) = -m}} \int_{0}^{\infty}   M_{-\frac{k}{2},s-\frac{1}{2}}(4 \pi m v) v^{-\frac{3}{2}-\frac{k}{2}}   e^{-2 \pi  v\left(q(X_z) - q\left(X_{z^\perp}\right)\right)} dv.
	\end{align*}
The integral is a Laplace transform. Using equation (11) on page 215 of \cite{IntegralTransforms} we obtain
\begin{multline*}
2(4\pi m)^{s+j-\frac{k}{2}} \frac{\Gamma\left(s+\frac{k}{2}\right) \Gamma(s-\frac{1}{2}-\frac{k}{2})}{\Gamma(2s) \Gamma\left(s+\frac{k}{2}-j\right)} \\
\times \sum_{\substack{X \in L+\mu \\ q(X) = -m}} \left(-4 \pi q(X_{z^\perp})\right)^{\frac{k}{2}+\frac{1}{2}-s} {}_2F_1\left(s-\frac{1}{2}-\frac{k}{2}, s+\frac{k}{2};2s;\frac{-m}{q(X_{z^\perp})}\right).
\end{multline*}
To obtain the statement of the theorem, we plug in the point $s = 1-\frac{k}{2}+j$ and use that $f$ may be written as a linear combination of the Maass--Poincar\'{e} series as in Section \ref{Section: Maass--Poincare series}. Convergence follows similarly to that of \cite[Theorem 3.1]{BS}.
\end{proof}

\begin{remark}
For small values of $n$ and $j$ the Gauss hypergeometric function on the right-hand side can be evaluated explicitly. For example, choose $n=1$. Then by \cite[Eq.~15.4.17]{NIST} and using that ${}_2F_1$ is symmetric in the first two arguments, we have that
	\begin{align*}
	{}_2F_1\left( \frac{1}{2}+j, 1+j;2+2j;\frac{q(X)}{q(X_{z^\perp})}\right) = \left(\frac{1}{2}+\frac{1}{2}\left(1- \frac{q(X)}{q(X_{z^\perp})}\right)^{\frac{1}{2}}\right)^{-1-2j} .
	\end{align*}
	For other values of $n$ and $j$, one can use transformations of ${}_2F_1$ along with the contiguous hypergeometric functions to obtain a (lengthy) description in terms of linear combinations of more elementary functions. The coefficients in the linear combination are rational.
\end{remark}

\subsection{Evaluating the theta lift at special points}

We now evaluate the theta integral at special points. Recall that at a special point $w$ we have the positive-definite one-dimensional lattice $P = L \cap w$ and negative-definite $n$ dimensional lattice $N = L \cap w^\perp$. Recall that $\mathcal{G}_P$ denotes a harmonic Maass form of weight $\frac{3}{2}$ for $\rho_P$ that maps to $\Theta_P$ under $\xi_{\frac 32}$. We denote its holomorphic part by $\mathcal{G}_P^+$. For simplicity, we now assume that the input $f$ for the regularized theta lift is weakly holomorphic (if this is not the case, then one obtains another integral on the right-hand side). 
\begin{theorem}\label{Theorem: theta lift at CM points}
	Let $f \in M_{k-2j, L}^!$. We have 
	\begin{align*}	\Lambda^{\reg}_j \left(f, w \right)
	= \frac{\pi^{\frac{1}{2}} \Gamma\left(1+j\right)  }{2 \Gamma\left( \frac{3}{2} + j \right)} (4\pi)^{j}  \CT\left(\left\langle f_{P \oplus N}(\tau) ,  \left[ \mathcal{G}_{P}^+(\tau), \Theta_{N^{-}}(\tau)\right]_{j}  \right\rangle \right).
	\end{align*}
\end{theorem}

\begin{remark}
	For fixed $j$, one can evaluate the constant term as sums over lattice vectors as in \cite[Remark 3.6]{BS}. However, the terms arising from the Rankin--Cohen bracket quickly become unwieldy for large values of $j$.
\end{remark}

\begin{proof}[Proof of Theorem \ref{Theorem: theta lift at CM points}]
	The proof is similar to those of \cite[Theorem 5.4]{bruinier2020greens} and \cite[Theorem 4.1]{ANBMS} and so we only sketch the details, for the convenience of the reader. We infer from Lemma~\ref{Lemma: restriction and trace maps} and \eqref{Equation: trace map for Theta_L} that
	\begin{equation*}
	\langle f, \Theta_L \rangle = \left\langle f, (\Theta_{P \oplus N})^L\right\rangle = \langle f_{P \oplus N}, \Theta_{P\oplus N}\rangle.
	\end{equation*}
	This means that one may assume that $L = P \oplus N$ if $f$ is replaced by $f_{P \oplus N}$. For succinctness, we write $f$ instead of $f_{P \oplus N}$ throughout.
	
	The first step is to use the self-adjointness of the raising operator given in \cite[Lemma~4.2]{bruinier2004borcherds} to obtain
	\begin{align*}
	\int_{\mathcal{F}}^{\reg} \left\langle R_{k-2j}^{j} (f)(\tau) , \overline{\Theta(\tau,w)} \right\rangle v^{k} d\mu(\tau) 
	= & (-1)^{j} \int_{\mathcal{F}}^{\reg} \left\langle  f(\tau) , R_{-k}^{j} \left(v^{k}\overline{\Theta_L(\tau,w)}\right) \right\rangle  d\mu(\tau),
	\end{align*}
	after noting that the apparent boundary term appearing disappears in the same way as in the proof of \cite[Lemma~4.4]{bruinier2004borcherds}. We next require the formula
	\begin{align}\label{Eqn: splitting}
	R_{\ell-\kappa}\left(v^{\kappa}\overline{g(\tau)} \otimes h(\tau)\right) = v^\kappa \overline{g(\tau)} \otimes R_\ell(h)(\tau)
	\end{align}
	which holds for every holomorphic function $g$, every smooth function $h$, and $\kappa,\ell \in \R$.
	Then using \eqref{Equation: splitting} and \eqref{Eqn: splitting} yields
	\begin{align*}
	R_{-k}^{j} \left( v^{-k} \overline{\Theta_{P \oplus N} (\tau,w)} \right) = L_\frac{3}{2} (\mathcal{G}_{P}) (\tau)\otimes R_{1-k}^{j} (\Theta_{N^{-}}) (\tau).
	\end{align*}

Next we use the fact that $L_1(\Theta_{N^-}) = 0$ along with Proposition~\ref{Proposition: lowering of RC brackets} together to infer that
	\begin{equation*}
	L_\frac{3}{2}(\mathcal{G}_P)(\tau) \otimes R_{1-k}^{j}(\Theta_{N^{-}}) (\tau) = \frac{\pi^{\frac{1}{2}} \Gamma\left(j+1\right)  }{2 \Gamma\left( \frac{3}{2} + j \right)} (-4\pi)^{j} L_{\frac{5}{2}-k+2j} \left(\left[ \mathcal{G}_P(\tau) , \Theta_{N^{-}}(\tau)\right]_{j}\right).
	\end{equation*}
	Therefore, overall we have
	\begin{multline*}
	\int_{\mathcal{F}}^{\reg} \left\langle R_{k-2j}^{j} (f)(\tau) , \overline{\Theta_L(\tau,w)} \right\rangle v^{k} d\mu(\tau) \\
	=  \frac{\pi^{\frac{1}{2}} \Gamma\left(j+1\right)  }{2 \Gamma\left( \frac{3}{2} + j \right)} (4\pi)^{j}  \int_{\mathcal{F}}^{\reg} \left\langle f(\tau),  L_{\frac{5}{2}-k+2j}\left(\left[ \mathcal{G}_P(\tau) , \Theta_{N^{-}}(\tau)\right]_{j}\right) \right\rangle d \mu(\tau).
	\end{multline*}
	Applying Stokes' Theorem as in the proof of \cite[Proposition~3.5]{bruinierfunke2004} yields the result, noting that apparent boundary terms arising vanish because $f$ is weakly holomorphic.
\end{proof}

\subsection{The Fourier expansion}
To end this section, we remark how one could obtain the Fourier expansion of the lift $\Lambda^{\reg}_j(f,z)$ for any $f \in H_{k-2j,L}^!$. This, however, would be a lengthy but completely clear calculation that we do not require in the present paper, and so we omit the details. 

Firstly, note that Lemma \ref{Lemma: R_k^n acting on F_m,h} implies that one can rewrite the lift as a linear combination of the lift
\begin{align*}
\int_{\mathcal{F}}^{\reg} \left\langle  F_{\mu,-m,k,s}\left(\tau\right), \overline{\Theta_L (\tau,z)} \right\rangle v^{k} d\mu(\tau).
\end{align*}
Recall that we evaluate at $s = 1-\frac{k}{2}+j$. In the case of $j=0$ the Fourier expansion of this lift was calculated by Bruinier in his celebrated habilitation \cite{bruinier2004borcherds}. One could then proceed as in the proof of \cite[Theorem 2.15 and Proposition 3.1]{bruinier2004borcherds} to obtain a Fourier expansion for $\Lambda^{\reg}(f,z)$ which is of a similar shape to the $j=0$ case. Finally, one could also translate this to the hyperbolic model of the Grassmanian in the same fashion as \cite[Theorem 3.3]{BS}.

	\section{Eichler-Selberg type relations}\label{Sec: mertens}
We are now in a position to use our techniques to offer proofs of Theorem \ref{Thm: MAIN} along with Theorems \ref{Thm: mertens} and \ref{Thm: AK}. Let $n=1$ throughout this section. 

\begin{proof}[Proof of Theorem \ref{Thm: MAIN}]
	For any $f \in H^{\text{cusp}}_{k-2j,L}$ by Theorem \ref{Theorem: lift as series} and the remark after it we have that
	\begin{align*}
	\Lambda_j^{\reg}(f,z) = 4^{1+2j} \pi^{\frac{1}{2}+j}  \frac{\Gamma(1+j)\Gamma\left(\frac{1}{2}+j\right)}{\Gamma(2+2j)} \sum_{\substack{X \in L'\\q(X)<0}} c_f^+(X,q(X)) \left(\sqrt{\lvert q(X_{z^\perp}) \rvert} - \sqrt{\lvert q(X_{z}) \rvert}\right)^{1+2j}.
	\end{align*}
	Simply combine this with Theorem \ref{Theorem: theta lift at CM points}, take the difference and apply Proposition \ref{Prop: Serre duality} to yield the statement of the theorem with $C_j = 2^{2j+3} \frac{\Gamma\left(\frac{1}{2}+j\right) \Gamma\left(\frac{3}{2}+j\right)}{\Gamma \left(2+2j\right)}$.
\end{proof}
Finally, we apply these results to see how one can obtain known results for scalar-valued modular forms from the higher Siegel theta lifts. Take the rational quadratic space $\Q^2$ with quadratic form $Q(a,b) =ab$. Then a lattice $L$ in $\Q^2$ has signature $(1,1)$ and is isotropic. For a special point $w$ we follow \cite{BS} and let $y=(y_1,y_2) \in L$ be its primitive generator with $y_1,y_2>0$. Let $y^\perp \in L$ be the primitive generator of $w^\perp$ with positive second coordinate. Then $y^\perp$ is a positive multiple of $(-y_1,y_2)$. Defining $d_P\coloneqq y_1y_2$ and $d_N \coloneqq y_1^\perp y_2^\perp$ we have $P \cong (\Z,d_P x^2)$ and $N \cong (\Z, -d_N x^2)$.
As in \cite{BS}, one may identify 
\begin{align*}
\Theta_P = \theta_{\frac{1}{2},d_P},\qquad \Theta_{N^-} = \theta_{\frac{1}{2},d_N}.
\end{align*}
\begin{proof}[Proof of Theorem \ref{Thm: mertens}]
	 We choose the lattice $L = (\Z, x^2) \oplus (\Z, -y^2)$ which we view as a sublattice of $(\R^2,xy)$. The lattice $L$ has a group ring isomorphic to $(\Z/2\Z)^2$ which we label by $\e_{(t,r)}$, and we choose the point $(y_1,y_2) =(1,1)$.
	
	Note the following explicit formulae (recalled from \cite{BS} for the convenience of the reader). We have
	\begin{align*}
	(X,y) = X_2y_1 + X_1y_2,\qquad -2X_1 X_2 |y^2| + (X,y)^2 = (X_1y_2 - X_2y_1)^2,
	\end{align*}
	along with
	\begin{align*}
	q(X_w) = \frac{1}{2} \frac{(X,y)^2}{|y|^2}, \qquad \lvert q(X_{w^\perp}) \rvert = -X_1 X_2 + \frac{1}{2} \frac{(X,y)^2}{|y|^2}.
	\end{align*}
	Combining these yields
	\begin{align*}
	\sqrt{\lvert q(X_{w^\perp}) \rvert} - \sqrt{\lvert q(X_w) \rvert} = \frac{1}{\sqrt{y_1y_2}} \min\left( \lvert X_1y_2 \rvert, \lvert X_2 y_1 \rvert \right).
	\end{align*}
	
	 Let $f$ be a weakly holomorphic modular form of appropriate weight. Then Theorems \ref{Theorem: lift as series} and \ref{Theorem: theta lift at CM points} along with the duplication formula for the $\Gamma$-function $\pi^{\frac{1}{2}}\Gamma(2x) = 2^{2x-1}\Gamma(x)\Gamma(x+\frac{1}{2})$ imply that
	\begin{align*}
	\frac{j! }{4\sqrt{\pi} \Gamma\left(\frac{1}{2}+j\right)} \CT\left( \langle f, [\mathcal{G}_P^+,\Theta_{N^-}]_j \rangle\right) = \sum_{\substack{X = (X_1,X_2) \in L' \\ X_1X_2<0}} c_f^+(X,X_1X_2) \min\left(\lvert X_1 \rvert, \lvert X_2 \rvert\right)^{2j+1}.
	\end{align*}
	
	 By results of Zagier \cite{zagier1975eisenstein}, the generating function
	 \begin{align}\label{eqn: H generating fn}
	 \mathcal{G}_P^+(\tau) = -8\pi \left(\sum_{\substack{n \geq 0 \\ n \equiv 0 \pmod 4}}H(n)e\left(\frac{n\tau}{4}\right)\right)\e_0-8\pi \left(\sum_{\substack{n \geq 0 \\ n \equiv 3 \pmod 4}}H(n)e\left(\frac{n\tau}{4}\right)\right)\e_1
	 \end{align}
	 is the holomorphic part of a harmonic Maass form $\mathcal{G}_P$ of weight $\frac{3}{2}$ for the dual Weil representation $\rho_{L^-}$ which maps to $\Theta_P$ under $\xi_{\frac{3}{2}}$. We plug this in for the choice of $\mathcal{G}_P^+$. Note that $\Theta_{N^-} = \theta_{\frac{1}{2},1}$ and that we pick up a minus sign from \eqref{eqn: H generating fn}.
	
	Concentrating on the Rankin--Cohen bracket, we have four components, given by
	\begin{align*}
 \left[\sum_{\substack{n \geq 0 \\ n \equiv -t^2 \pmod{4}}} H(n) e\left(\frac{n\tau}{4}\right), \sum_{ \substack{n \in \Z \\ n \equiv r \pmod{2}}} e\left(\frac{n^2\tau}{4}\right) \right]_j \e_{(t,r)},
	\end{align*}
	for $t,r \pmod{2}$.
	
	 Overall, after dividing the result by $2$ and letting $[\cdot,\cdot]_j^{[n]}$ denote the $n$-th coefficient in the Fourier expansion of the Rankin--Cohen bracket, a short simplification yields that 
	\begin{align*}
	& \frac{j! \sqrt{\pi}}{ \Gamma\left(\frac{1}{2}+j\right)}  \sum_{t,r \pmod{2}} \sum_{m \geq 0} c_f^+( (t,r),-m/4) \left[\sum_{\substack{n \geq 0 \\ n \equiv -t^2 \pmod{4}}} H(n) e\left(\frac{n\tau}{4}\right), \sum_{ \substack{n \in \Z \\ n \equiv r \pmod{2}}} e\left(\frac{n^2\tau}{4}\right) \right]_j^{[m/4]} \e_{(t,r)}
	\\ &+ \frac{1}{2} \sum_{\substack{t,r \pmod{2}}} \sum_{m \geq 0} \sum_{\substack{a,b \in \N \\ ab=m \\ a,b \equiv t+r \pmod{2}}} c_f^+((t,r),-m/4) \min\left(a,b\right)^{2j+1} \e_{(t,r)}
	\end{align*}
	vanishes. We then use Proposition \ref{Prop: Serre duality} to conclude that 
	\begin{align*}
	g(\tau) \coloneqq \sum_{\substack{t,r \pmod{2}}} \sum_{m \geq 0} c_g( (t,r), m)  q^{\frac{m}{4}} \mathfrak{e}_{t,r}
	\end{align*}
	where
	\begin{multline*}
	c_g((t,r),m) \coloneqq 
	 c_j \left[\sum_{\substack{n \geq 0 \\ n \equiv -t^2 \pmod{4}}} H(n) e\left(\frac{n\tau}{4}\right), \sum_{ \substack{n \in \Z \\ n \equiv r \pmod{2}}} e\left(\frac{n^2\tau}{4}\right) \right]_j^{[m/4]} \\
	 + \frac{1}{2} \sum_{\substack{a,b \in \N \\ ab=m \\ a,b \equiv t+r \pmod{2}}} \min\left(a,b\right)^{2j+1}
	\end{multline*}
	is a holomorphic modular form of weight $2j+2$ for $\rho_L$. Summing the components $\mathfrak{e}_{(1,0)}$ and $\mathfrak{e}_{(0,1)}$ and shifting $\tau \mapsto 4\tau$ yields precisely the expression in the statement of the theorem. Moreover, a direct computation using the Weil representation shows that this is a modular form of weight $2j+2$ on $\Gamma_0(4)$. That it is a cusp form is also clear.
\end{proof}

\begin{remark}
	The above calculations show that $g(\tau)$ has an even part given by the other two components labeled by $(0,0)$ and $(1,1)$. Tracing this through to the scalar-valued case, we obtain that
	\begin{align*}
	c_j[\mathcal{H},\vartheta]_j^{\text{even}} + \sum_{m \geq 0} \sum_{\substack{a,b \in \N \\ ab=4m}} \min(a,b)^{2j+1} q^{4m}
	\end{align*}
	is a holomorphic modular form of weight $2j+2$ on $\Gamma_0(4)$.
\end{remark}

	Further examples could also be computed by combining the techniques in \cite{BS} with those in the current paper. As remarked previously, for small values of $j$, these can be written explicitly, but for larger values of $j$ the terms arising from the Rankin--Cohen brackets quickly become unwieldy. With $\boldsymbol{x} =(x_1, \dots, x_j) \in \Z^j$, $p$ some polynomial and $Q$ a quadratic form, one could obtain formulae for sums of the shape
	\begin{align*}
	\sum_{\boldsymbol{x} \in \Z^j} p(\boldsymbol{x}) H(t-Q(\boldsymbol{x}))
	\end{align*}
	in terms of linear combinations (with rational coefficients) of divisor power sums. This complements the large literature on recurrences for Hurwtiz class numbers. It is also clear that one could derive similar examples for other generating functions of weight $\frac{3}{2}$.

\begin{proof}[Proof of Theorem \ref{Thm: AK}]
	The proof is very similar to that of Theorem \ref{Thm: mertens}, and so we only offer a sketch and leave the details to the interested reader. We choose the lattice $L=(\Z,6x^2) \oplus(\Z,-6y^2)$ viewed as a sublattice of $(\R^2,xy)$ as before. Then $L$ has a group ring isomorphic to $\Z/12\Z$. Recall that
	\begin{align*}
	g(\tau) = q^{-\frac{1}{24}} \sum_{n \geq 0} \left( \spt\left(n\right) +\frac{1}{12}(24n-1)p(n) \right) q^{n}.
	\end{align*}
	It is known that $g$ is the holomorphic part of weight $\frac{3}{2}$ harmonic Maass form $F$ on $\Gamma_0(576)$ that maps to $(-\sqrt{6}/4) \eta$ under $\xi_{\frac{3}{2}}$, see \cite{Bri}. As in \cite{BS}, we see that
	\begin{align*}
	\sum_{r \pmod{12}} \leg{12}{r} F(\tau) \mathfrak{e}_r = F(\tau)(\e_1 -\e_5-\e_7+\e_{11})
	\end{align*}
	is a harmonic Maass form of weight $\frac{3}{2}$ for the Weil representation of the lattice $(\Z,6x^2)$. Moreover, under the action of $\xi_{\frac{3}{2}}$ it maps to
	\begin{align*}
	-\frac{\sqrt{6}}{4} \left( \theta_{\frac{1}{2},6} - \theta_{\frac{1}{2},6}^{\sigma_2} \right).
	\end{align*}
	
One evaluates the lift at the special points $y=(1,6)$ and $y=(2,3)$ and takes the difference. As before, this results in two different evaluations, one in terms of a series of minima, and one as a constant term in a Fourier expansion. After shifting the constants to the $\CT$ term, one proceeds as in the proof of Theorem \ref{Thm: mertens} to use Proposition \ref{Prop: Serre duality} to obtain a modular form of weight $2j+2$ whose coefficients are essentially a rescaling of those stated in the theorem. Making use of the fact that $\eta$ may be written as the difference of two unary theta functions yields precisely the function in the statement of the theorem (up to rescaling by a constant factor, and after rewriting the sum over minima by case distinctions and using the duplication formula for the $\Gamma$-function twice). This time, one does not need to shift $\tau$, and so obtains a modular form on $\SL_2(\Z)$.
\end{proof}

\end{document}